\documentclass[11pt,a4paper]{article}

\usepackage{booktabs}
\usepackage{amssymb,latexsym}
\usepackage{a4}

\newcommand{\rb}[1]{\raisebox{1.5ex}[-1.5ex]{#1}}
\newcommand{\ra}{\rangle}
\newcommand{\la}{\langle}
\newcommand{\X}{{\mathcal X}}
\newcommand{\Q}{{\mathcal Q}}
\newcommand{\R}{{\mathcal R}}

\newcommand{\Z}{{\mathbb Z}}
\newcommand{\N}{{\mathbb N}}
\newcommand{\rk}{\mathrm{rk}}
\newcommand{\End}{{\mathrm{End}}}
\newcommand{\hirsch}{{\mathfrak h}}
\newcommand{\mc}{\mathcal}

\newcommand{\NQL}{{\sf NQL}}

\newcommand{\GAP}{{\sf GAP}}
\newcommand{\Grig}{{\mathfrak G}}
\newcommand{\FG}{\Gamma}
\newcommand{\BSV}{{\mathrm{BSV}}}
\newcommand{\ti}{\tilde}

\title{Approximating the Schur multiplier of certain infinitely presented groups via
nilpotent quotients}
\author{Ren\'e Hartung}
\date{April 28, 2010}

\newenvironment{proof}{\par\vskip-\lastskip\vskip\topsep
                       \noindent{\it Proof.}\vadjust{\nobreak}\quad
                       \begingroup\divide\topsep3\divide\itemsep3
                       \divide\partopsep3\divide\parskip3
                       \divide\parsep3}
                      {\ifvmode\penalty10000\hbox to\hsize{\hfil$\Box$}
                       \else\parfillskip0pt\widowpenalty10000\hfil$\Box$
                       \fi\par\vskip 1.5ex\endgroup}

\newtheorem{theorem}{Theorem}
\newtheorem{lemma}[theorem]{Lemma}
\newtheorem{example}[theorem]{Example}
\newtheorem{conjecture}{Conjecture}
\newtheorem{proposition}{Proposition}

\renewcommand{\theconjecture}{\Alph{conjecture}}

\begin{document}
\maketitle
\begin{abstract}
  We describe an algorithm for computing successive quotients of the Schur
  multiplier $M(G)$ for a group $G$ given by an invariant finite 
  $L$-presentation. As application, we investigate the Schur multipliers
  of various self-similar groups such as the Grigorchuk 
  super-group, the generalized Fabrykowski-Gupta groups, the Basilica 
  group and the Brunner-Sidki-Vieira group.\bigskip

\end{abstract}

\section{Introduction}
The Schur multiplier $M(G)$ of a group $G$ can be defined as the
second homology group $H_2(G,\Z)$. It was introduced by Schur and is,
for instance, relevant in the theory of central group extensions. In
combinatorial group theory, the Schur multiplier found its applications
due to the Hopf formula: if $F$ is a free group and $R$ is a normal
subgroup of $F$ so that $G\cong F/R$, then the Schur multiplier of $G$
is isomorphic to the factor group $(R\cap F')/[R,F]$. For further details
on the Schur multiplier we refer to~\cite[Chapter~11]{Rob96}.

The Hopf formula yields that every finitely presentable group has
a finitely generated Schur multiplier. This is used in~\cite{Gri99}
for proving that the Grigorchuk group is not finitely presentable: its
Schur multiplier is infinitely generated $2$-elementary abelian. This
answers the questions in~\cite{Bau93} and~\cite{Se91}. There are various
examples of self-similar groups other than the Grigorchuk group for which
it is not known whether their Schur multiplier is finitely generated or
whether the groups are finitely presented.

The first aim of this paper is to introduce an algorithm for investigating
the Schur multiplier of self-similar groups with a view towards its 
finite generation. Let $G$ be a group with a presentation $G \cong F / R$.
Then $G/\gamma_cG \cong 
F/R\gamma_cF$, where $\gamma_cG$ is the $c$-th term of the lower
central series of $G$. We identify $M(G)$ with $(R \cap F')/[R,F]$ and
$M(G/\gamma_cG)$ with $(R\gamma_cF\cap F')/[R\gamma_cF,F]$ and define
\[
  \varphi_c\colon M(G)\to M(G/\gamma_cG),\: g[R,F]\mapsto g[R\gamma_cF,F].
\]

Then $\varphi_c$ is a homomorphism of abelian groups. We describe an 
effective method to determine the \emph{Dwyer quotients} $M_c(G) = M(G)/\ker\varphi_c$,
for $c \in \N$, provided that $G$ is given by an invariant finite 
$L$-presentation, see \cite{Bar03,BEH08} or Section~\ref{sec:Prelim} below. 
Every finitely presented group and many self-similar groups can be described 
by a finite invariant $L$-presentation. An implementation of our algorithm
is available in the \NQL-package~\cite{NQL} of the computer algebra system 
\GAP; see~\cite{GAP}.  

We have applied our algorithm to various examples of self-similar groups:
the Grigorchuk super-group $\ti\Grig$, see~\cite{BG02}, the Basilica group 
$\Delta$, see~\cite{GZ02a,GZ02b}, the Brunner-Sidki-Vieira group 
$\BSV$, see~\cite{BSV99}, and some generalized Fabrykowski-Gupta groups 
$\FG_d$, see~\cite{FG85,Gri00}. As a result, we observed that the sequence 
$(M_1(G), \ldots, M_c(G), M_{c+1}(G), \ldots)$ exhibits a periodicity 
in $c$ in all these cases. Based on this, we propose the following 
conjecture.

\renewcommand{\theconjecture}{\Roman{conjecture}}
\begin{conjecture} 
$\mbox{}$
\begin{itemize}\addtolength{\itemsep}{-0.75ex}
\item $M_c(\ti\Grig)$ is $2$-elementary abelian of rank 
      $2\lfloor \log_2(c) \rfloor + 2\lfloor
      \log_2 \frac{c}{3}\rfloor + 5$, for $c \geq 4$.
\item $M_c(\Delta)$ has the form $\Z^2\times{\mathcal A}_c$, where
      ${\mathcal A}_c$ is an abelian $2$-group of rank
      $\lfloor\log_2\frac{c}{3}\rfloor$ and exponent $2^{2\lfloor
      \frac{c-6}{2}\rfloor+2}$, for $c\geq 6$.
\item $M_c(\BSV)$ has the form $\Z^2\times{\mathcal B}_c$, where
      ${\mathcal B}_c$ is an abelian $2$-group
      of rank $\lfloor\log_2\frac{c}{5}\rfloor +
      \lfloor\log_2\frac{c}{9}\rfloor+3$ and exponent
      $2^{2\lfloor\frac{c-4}{2} \rfloor+1}$, for $c\geq 4$.
\item For a prime power $d$, the group $M_c(\FG_d)$ has exponent
      $d$ for $c$ large enough; its rank is an increasing function in $c$
      which exhibits a periodic pattern.
\end{itemize}
\noindent In particular, all of these groups have an infinitely generated
Schur multiplier and are therefore not finitely presentable.
\end{conjecture}
\setcounter{conjecture}{0}
\renewcommand{\theconjecture}{\Alph{conjecture}}

Further details on the periodicities and the computational evidence for
them are given in Section~\ref{sec:Apps}.

\section{Preliminaries}\label{sec:Prelim}
In the following, we recall the basic notion of invariant and finite
$L$-presentations and the basic theory of the Schur multiplier of a group.
Let $F$ be a finitely generated free group over the alphabet $\X$. Further
suppose that $\Q,\R\subset F$ are finite subsets of the free group
$F$ and $\Phi\subset\End(F)$ is a finite set of endomorphisms of $F$.
Then the quadruple $\la\X\mid\Q\mid\Phi\mid\R\ra$ is a \emph{finite
$L$-presentation}. It defines the \emph{finitely $L$-presented group}
\[
  G = \left\la \X~\middle|~\Q\cup \bigcup_{\varphi\in\Phi^*} \R^\varphi
  \right\ra,
\]
where $\Phi^*$ denotes the free monoid generated by $\Phi$; that is,
the closure of $\Phi\cup\{{\rm id}\}$ under composition. A finite
$L$-presentation $\la\X\mid\Q\mid\Phi\mid\R\ra$ is \emph{invariant} if
every endomorphism $\varphi\in\Phi$ induces an endomorphism of $G$; that
is, if the normal closure of $\Q\cup\bigcup_{\varphi\in\Phi^*}\R^\varphi$
in $F$ is $\varphi$-invariant. For example, every finite $L$-presentation
of the form $\la\X\mid\emptyset\mid\Phi\mid\R\ra$ is invariant. Clearly,
invariant finite $L$-presentations generalize finite presentations since
every finitely presented group $\la\X\mid\R\ra$ is finitely $L$-presented
by $\la\X\mid\emptyset\mid\{{\rm id}\}\mid\R\ra$. Further examples of
invariantly $L$-presented groups are several self-similar groups including
the Grigorchuk group~\cite{Gri80}, the Basilica group~\cite{GZ02a,GZ02b},
and the Brunner-Sidki-Vieira group~\cite{BSV99}.\medskip

In the remainder of this section, we recall the basic theory of the
Schur multiplier of a group $G$. Recall that, in general, the
Schur multiplier of a finitely presented group is not computable;
see~\cite{Gor95}. But, for instance, if $G$ is finite, then $M(G)$
can be deduced from a finite presentation of $G$ with the Hopf formula
and the Reidemeister-Schreier algorithm. A more effective algorithm for
finite permutation groups is described in~\cite{Ho84}. Recently, Eick
and Nickel~\cite{EN08} described an algorithm for computing the Schur
multiplier of a polycyclic group given by a polycyclic presentation.

Let $F$ be a free group and $R$ be a normal subgroup of $F$ so that
$G\cong F/R$. Then the Hopf formula gives
\begin{equation}\label{eqn:Hopf}
  M(G) \cong (R\cap F') / [R,F].
\end{equation}
Suppose that $N$ is a normal subgroup of $G$ and let $S$ be a normal
subgroup of $F$ so that $SR/R$ corresponds to $N$. Then Blackburn and
Evens~\cite{BE79} determined the exact sequence
\[
  1 \to (R\cap [S,F])/([R,F]\cap [S,F])\to M(G) \to
  M(G/N) \to(N\cap G')/[N,G]\to 1.
\]
Applying this sequence to the lower central series term $N=\gamma_cG$
yields the exact sequence
\[
  1 \to (R\cap \gamma_{c+1}F)/([R,F]\cap\gamma_{c+1} F)
  \to M(G) \stackrel{\varphi_c}{\to} M(G/\gamma_cG) \to \gamma_cG / \gamma_{c+1} G \to 1.
\]
This gives a filtration $M(G) \geq \ker\varphi_1 \geq \ker\varphi_2
\geq\ldots$, called the \emph{Dwyer-filtration}, of the Schur
multiplier of $G$. Note that, if $G$ has a maximal nilpotent quotient
of class $c$, then
\[
  \bigcap_{c\in\N_0} \ker\varphi_c \cong 
  (R\cap\gamma_{c+1}F)[R,F]/[R,F].
\]
However, even if the group $G$ is residually nilpotent, the group
$F/[R,F]$ is not necessarily residually nilpotent; see~\cite{Mik05}
and~\cite{Blu07}. Thus the \emph{Dwyer-kernel} $\bigcap_{c\in\N}
\ker\varphi_c$ is possibly non-trivial.\medskip

We note that the Schur multiplier $M(G/\gamma_cG)$ can be computed
with the algorithm in~\cite{EN08} while the isomorphism type
of $\gamma_cG/\gamma_{c+1}G$ can be computed with the nilpotent
quotient algorithm in~\cite{BEH08}. Therefore, the sequence $M(G)\to
M(G/\gamma_cG) \to \gamma_cG/\gamma_{c+1}G \to 1$ allows to determine
the size of $M_c(G)$ provided that $M(G/\gamma_cG)$ is finite. However,
the algorithm described here determines the structure of $M_c(G)$ even
if the Schur multiplier $M(G/\gamma_cG)$ is infinite.

\section{Adjusting an invariant $L$-presentation}\label{sec:ModLp}
In order to prove the following theorem, we explicitly describe
an algorithm for modifying an invariant $L$-pre\-sen\-ta\-tion. The
resulting $L$-presentation enables us to read off a generating set for the
Schur multiplier in Section~\ref{sec:CovMG}. Our algorithm generalizes
the explicit computations in~\cite{Gri99}.
\begin{theorem}\label{thm:ModLp}
  Let $\la\X\mid\Q\mid\Phi\mid\R\ra$ be an invariant finite
  $L$-presentation which defines the group $G=F/R$. Then $G$ admits
  an invariant finite $L$-presentation $\la\X\mid\Q'\cup{\mc
  B}\mid\Phi\mid\R'\ra$ with $\Q',\R'\subset F'$ and ${\mc B}\subset F$
  satisfying $|{\mc B}| = |\X|-\hirsch(G/G')$, where $\hirsch(G/G')$
  denotes the torsion-free rank of $G/G'$.
\end{theorem}
\begin{proof}
  Since $\la\X\mid\Q\mid\Phi\mid\R\ra$ is an invariant $L$-presentation,
  every endomorphism $\varphi\in\Phi$ induces an endomorphism of
  the group $G$. Thus we have $R^\varphi\subseteq R$, for every
  $\varphi\in\Phi^*$. In particular, every image of a relator in
  $\Q\cup\R$ is a consequence; that is, $\Q^\varphi\subset R$ and
  $\R^\varphi\subset R$, for every $\varphi\in\Phi^*$.

  Write $n=\rk(F)$. Then the abelianization $\pi\colon F\to\Z^n$
  maps every $x\in F$ to its corresponding exponent vector
  $a_x\in\Z^n$. Clearly, $\ker\pi=F'$ and, since $F'$ is fully-invariant,
  every $\varphi\in\Phi$ induces an endomorphism of the free abelian
  group $\Z^n$. Therefore, the exponent vector of $x^\varphi$ is the
  image $a_xM_\varphi$ for some matrix $M_\varphi\in\Z^{n\times n}$. Now,
  the normal subgroup $RF'$ maps onto
  \begin{equation}
    U = \la a_q, a_r M_\varphi\mid q\in\Q, r\in\R,\varphi\in\Phi^*\ra\leq\Z^n.
    \label{eqn:UZn}
  \end{equation}
  As every subgroup of $\Z^n$ is generated by at most $n$ elements,
  the subgroup $U$ is finitely generated. In the following, we may
  use the spinning algorithm from~\cite{BEH08} and Hermite normal form
  computations to compute a basis for the subgroup $U$ while modifying
  the $L$-presentation simultaneously.

  Let $B$ be a basis of $\la a_q\mid q\in\Q\ra$.  Then every element
  $u\in B$ is a $\Z$-linear combination of elements in $\{a_q\mid
  q\in\Q\}$. Hence, for every $u\in B$, there exists a word $r_u$ in the
  relators in $\Q$ such that $a_{r_u} = u$. Define ${\mc B}=\{r_u\mid
  u\in B\}$. Then, for every $q\in\Q$, it holds that $a_q\in\la B\ra$ as
  $B$ is a basis and hence, there exists a word $w_q$ in the $r_u$'s so
  that $a_{w_q}=a_q$. Define $\Q' = \{q w_q^{-1}\mid q\in\Q\}$. Then the
  exponent vector of each element in $\Q'$ vanishes and hence $\Q'\subset
  F'$. Moreover, the invariant and finite $L$-presentation
  \[
    \la\X\mid\Q' \cup {\mc B} \mid \Phi\mid \R\ra
  \]
  still defines the group $G$ as we only applied Tietze transformations
  to the given $L$-presentation.

  It remains to force the elements of $\R$ into the derived subgroup
  $F'$. For this purpose, we will use the spinning algorithm
  from~\cite{BEH08} as follows: Initialize $\R'=\emptyset$.  As long
  as $\R$ is non-empty, we take an element $r\in\R$ and remove it from
  $\R$. Then either $a_r\in\la B\ra$ or $a_r\not\in\la B\ra$ holds. If
  $a_r\in\la B\ra$, then there exists a word $w_r$ in the $r_u$'s such
  that $a_{w_r}=a_r$ and hence, $rw_r^{-1}\in F'$. In this case we just
  add $rw_r^{-1}$ to $\R'$. Note that, for every $\varphi\in\Phi^*$, the
  word $(w_r^{-1})^\varphi$ is a consequence and hence, we can replace the
  relator $r^\varphi$ in the $L$-presentation by $(rw_r^{-1})^\varphi$.
  The invariant and finite $L$-presentation
  \[ 
    \la\X\mid\Q'\cup{\mc B} \mid \Phi\mid \R' \cup \R\ra
  \]
  still defines the group $G$.

  If, on the other hand, $a_r\not\in\la B\ra$ holds, we enlarge the
  current basis $B$ and modify the set ${\mc B}$. Let $B'$ be a basis
  for $\la B\cup\{a_r\}\ra$. Then every $v\in B'$ is a $\Z$-linear
  combination of the elements in $B\cup\{a_r\}$ and hence, there
  exists a word $\ti r_v$ in ${\mc B}\cup\{r\}$ such that $a_{\ti
  r_v} = v$. Define ${\mc B} = \{\ti r_v\mid v\in B'\}$. Then, by
  construction, either $|{\mc B}|=|B|+1$ or $|{\mc B}|=|B|$ holds. In
  the latter case, there is an element $u\in B$ so that $u\in\la
  (B\setminus\{u\})\cup\{a_r\}\ra$ holds. Thus, there exists a word
  $w_u$ in the elements of ${\mc B}$ such that $a_{w_u}=u$ and hence,
  $r_u w_u^{-1}\in F'$. In this case, we add $r_uw_u^{-1}$ to $\Q'$
  and add the images $\{r_u^\varphi\mid\varphi\in\Phi\}$ to $\R$. This
  yields an invariant and finite $L$-presentation $\la\X\mid\Q'\cup{\mc
  B}\mid\Phi\mid\R'\cup\R\ra$, with $\Q',\R'\subset F'$, which still
  defines the group $G$.

  As ascending chains of subgroups in $\Z^n$ terminate, eventually every
  exponent vector of an element in $\R$ is contained in the subgroup
  $\la B\ra$ and hence, the algorithm described above eventually
  terminates.  Clearly, the basis $B$ is then a basis for the subgroup
  $U$ in~(\ref{eqn:UZn}). As shown in~\cite{BEH08}, the abelian quotient
  $G/G'$ is isomorphic to the factor $\Z^n/U$. Its torsion-free rank is
  $n-|B|$ as claimed above.
\end{proof}
In the following example, we recall the explicit computations
in~\cite{Gri99} for the Grigorchuk group $\Grig$.
\begin{example}\label{ex:ModGri}
  Consider the Grigorchuk group $\Grig$ with its invariant $L$-presentation
  \[
    \Grig \cong \la\{a,b,c,d\}\mid\{a^2,b^2,c^2,d^2,bcd\}\mid\{\sigma\}
                              \mid\{(ad)^4, (adacac)^4\}\ra
  \]
  where $\sigma$ is the free group endomorphism induced by the mapping
  \[
    \sigma\colon ~\left\{
    \begin{array}{rcl}
      a &\mapsto& c^a\\
      b &\mapsto& d\\
      c &\mapsto& b\\
      d &\mapsto& c.
    \end{array}\right.
  \]
  As the exponent vectors $(2,0,0,0)$, $(0,1,1,1)$, $(0,0,2,0)$,
  and $(0,0,0,2)$ of the relations $a^2$, $bcd$, $c^2$, and $d^2$,
  respectively, are $\Z$-linearly independent forming a basis for the
  subgroup $U$ in (\ref{eqn:UZn}), we can modify this presentation so
  that the relations become
  \begin{equation}\label{eqn:ModGri}
    a^2,c^2,d^2,bcd,b^2(bcd)^{-2}c^2d^2,
     \sigma^k( (ad)^4 a^{-4} d^{-4} ), 
     \sigma^k( (adacac)^4 a^{-12}c^{-8}d^{-4}),
  \end{equation}
  for every $k\in\N_0$. Since the $L$-presentation is invariant,
  the images $\sigma^k(a^{-4}d^{-4})$ and $\sigma^k( a^{-12}c^{-8}d^{-4}
  )$ are consequences. Hence, the invariant finite $L$-pre\-sen\-ta\-tion
  \[
    \la\{a,b,c,d\}\mid\{b^2(bcd)^{-2}c^2d^2\}\cup\{a^2,c^2,d^2,bcd\}\mid
                      \{\sigma\}\mid\R'\ra,
  \]
  where $\R'=\{(ad)^4 a^{-4} d^{-4},(adacac)^4 a^{-12}c^{-8}d^{-4}\}$,
  defines the Grigorchuk group $\Grig$ and, as $\Grig/\Grig'\cong \Z_2^3$,
  it has the form as claimed in Theorem~\ref{thm:ModLp}.
\end{example}

\section{A generating set for the Schur multiplier}\label{sec:CovMG}
Let $G$ be a finitely generated group. We will use the results of
Theorem~\ref{thm:ModLp} and the Hopf formula to give a generating
set for the Schur multiplier of $G$ if $G$ is invariantly finitely
$L$-presented. Suppose that $F$ is a finitely generated free group and
$R$ is a normal subgroup of $F$ so that $G\cong F/R$. Then $F/[R,F]$ is a
central extension of $R/[R,F]$ by the group $G$ and the subgroup $R/[R,F]$
contains $(R\cap F')/[R,F]$. By the Hopf formula, the latter subgroup is
isomorphic to the Schur multiplier of $G$. Further, the subgroup $R/[R,F]$
decomposes as follows.
\begin{proposition}\label{prop:Mult}
  Let $G\cong F/R$ with a finitely generated free group $F$. Then we 
  have that
  \[
    R/[R,F] \cong \Z^{\rk(F) - \hirsch(G/G')} \oplus M(G).
  \]
\end{proposition}
\begin{proof}
  The factor $RF'/F'$ is free abelian with torsion-free rank
  $\rk(F)-\hirsch(G/G')$.  Since $RF'/F'\cong R/(R\cap F')$ is free
  abelian, the subgroup $(R\cap F')/[R,F]$ has a free abelian complement
  of rank $\rk(F)-\hirsch(G/G')$ and thus, the central subgroup $R/[R,F]$
  decomposes as claimed above.
\end{proof}
As $R/[R,F]$ is central in $F/[R,F]$, it is generated by the images of
the normal generators of $R$. Thus, in particular, if $R$ is finitely
generated as normal subgroup (that is, if $G$ is finitely presentable),
then $R/[R,F]$ is a finitely generated abelian group and so is its
subgroup $(R\cap F')/[R,F]$.

If $G$ is finite, then $R/[R,F]$ is an abelian subgroup with finite index
in $F/[R,F]$. A finite presentation for $F/[R,F]$ can be obtained from
a finite presentation of $G$. Then the Reidemeister-Schreier algorithm
yields a finite presentation for $R/[R,F]$ from which the isomorphism
type of $M(G)$ is obtained easily.

If $G$ is polycyclic, then it is finitely presentable and hence,
the group $F/[R,F]$ is an extension of a finitely generated abelian
group by a polycyclic group. In particular, $F/[R,F]$ is polycyclic
in this case. A consistent polycyclic presentation for $F/[R,F]$
can be computed with the algorithm in~\cite{EN08}. This polycyclic
presentation enables us to read off the isomorphism type of $R/[R,F]$ and,
by Proposition~\ref{prop:Mult}, the isomorphism type of $M(G)$. If $G$
is finitely generated and nilpotent of class $c$, then $F/[R,F]$ is
nilpotent of class at most $c+1$.  If $G$ is given by a weighted nilpotent
presentation, then the algorithm in~\cite{Nic96} computes a weighted
nilpotent presentation for $F/[R,F]$.

We now consider the case of an invariantly finitely $L$-presented
group $G$. Even though its Schur multiplier is not computable in
general, the following theorem yields a generating set for
$M(G)$ as subgroup of $R/[R,F]$.
\begin{theorem}\label{thm:GensMG}
  Let $\la\X\mid\Q'\cup{\mc B}\mid\Phi\mid\R'\ra$ be an invariant
  finite $L$-presentation of $G$ as provided by
  Theorem~\ref{thm:ModLp}. Further let $\pi\colon F\to F/[R,F],
  x\mapsto \bar x$ denote the natural homomorphism. Then we have that
  \[
    M(G)\cong \la \bar q, \overline{r^\varphi} \mid q\in\Q', r\in\R',
                                                    \varphi\in\Phi^*\ra.
  \]
\end{theorem}
\begin{proof}
  Clearly, $R/[R,F]$ is generated by the images of $\Q'\cup{\mc
  B}\cup\bigcup_{\varphi\in\Phi^*} (\R')^\varphi$. As the relators in
  $\Q'\cup\R'$ are contained in $F'$, it holds that
  \begin{equation}\label{eqn:GensMG}
    \{ \bar q,\overline{r^\varphi} \mid q\in\Q',r\in\R',\varphi\in\Phi^*\}
    \subseteq (R\cap F')/[R,F].
  \end{equation}
  We are left with the relators in ${\mc B}$. Recall
  that we have $|{\mc B}|=\rk(F)-\hirsch(G/G')$. Hence, the images
  $\{\bar r\mid r\in{\mc B}\}$ generate a free abelian complement
  to the Schur multiplier $(R\cap F')/[R,F]$ in $R/[R,F]$. Therefore,
  the images in~(\ref{eqn:GensMG}) necessarily generate 
  $(R\cap F')/[R,F]$.
\end{proof}
As the group $G$ in Theorem~\ref{thm:GensMG} is invariantly $L$-presented,
for every endomorphism $\varphi\in\Phi$, we have $R^\varphi\subseteq R$
and $[R,F]^\varphi \subseteq [R,F]$. Therefore, every $\varphi\in\Phi$
also induces an endomorphism of $F/[R,F]$ which fixes the subgroup
$R/[R,F]$. Further, as $F'$ is fully-invariant, every such $\varphi$
induces an endomorphism $\bar\varphi$ of $(R\cap F')/[R,F]$. This
yields that
\[
  M(G)\cong \la \bar q, {\bar r\,}^{\bar\varphi} \mid q\in\Q', r\in\R',
                                                      \varphi\in\Phi^*\ra
\]
and hence, the free monoid $\Phi^*$ induces a $\Phi^*$-module structure on the
Schur multiplier $M(G)$ in a natural way:
\begin{lemma}\label{lem:ModStruc}
  Let $\la\X\mid\Q\mid\Phi\mid\R\ra$ be an invariant finite
  $L$-presentation. Then the Schur multiplier $M(G)$ is finitely
  generated as a $\Phi^*$-module.
\end{lemma}
In particular, the Schur multiplier $M(G)$ has the form $A\oplus
\bigoplus_{\Phi^*} B$ with finitely generated abelian groups $A$ and $B$;
see~\cite{Bar03}.\medskip

We proceed with Example~\ref{ex:ModGri} by describing a generating
set for the Schur multiplier of the Grigorchuk group as provided by
Theorem~\ref{thm:GensMG}; cf.~\cite{Gri99}.
\begin{example}
  Consider the invariant finite $L$-presentation of the Grigorchuk group
  $\Grig$ as determined in Example~\ref{ex:ModGri}. Then the images of
  \begin{equation}\label{eqn:SchuGri}
    b^2 (bcd)^{-2} c^2d^2, 
    \sigma^k( (ad)^4 a^{-4} d^{-4} ), 
    \sigma^k( (adacac)^4 a^{-12}c^{-8}d^{-4}),~\textrm{ with }k\in\N_0,
  \end{equation}
  in $F/[R,F]$, generate the subgroup $(R\cap F')/[R,F]$. The images
  in $F/[R,F]$ of the relations $a^2$, $c^2$, $d^2$, and $bcd$ generate
  a free abelian complement to the Schur multiplier $(R\cap F')/[R,F]$
  in $R/[R,F]$.
\end{example}

\section{Approximating the Schur multiplier}\label{sec:FGQuot}
We finally describe our algorithm for approximating the Schur
multiplier of an invariantly finitely $L$-presented group $G$. Let
$\la\X\mid\Q\mid\Phi\mid\R\ra$ be an invariant finite $L$-presentation
defining the group $F/R$ so that $G\cong F/R$.  Then $G$ is finitely
generated and hence, its lower central series quotient $G/\gamma_cG$
is polycyclic. The nilpotent quotient algorithm
in~\cite{BEH08} computes a weighted nilpotent presentation for
$G/\gamma_cG$ together with the natural homomorphism $\pi\colon F\to
G/\gamma_cG$. In~\cite{Nic96}, Nickel described a covering-algorithm
which, given a weighted nilpotent presentation for $G/\gamma_cG$
and the homomorphism $\pi$, computes a polycyclic presentation for
$F/[R\gamma_cF,F]$ together with the natural homomorphism $\bar\pi\colon
F\to F/[R\gamma_cF,F]$.  The homomorphism $\bar\pi$ induces the
homomorphism $\varphi_c\colon M(G)\to M(G/\gamma_cG)$ as follows:
By Theorem~\ref{thm:ModLp}, the group $G$ has an invariant finite
$L$-presentation of the form
\[
 \la\X\mid\Q'\cup{\mc B}\mid\Phi\mid\R'\ra,\quad\textrm{with }\Q',\R'\subset F'
\]
and $|{\mc B}| = |\X|-\hirsch(G/G')$. Now, by Theorem~\ref{thm:GensMG},
the images of $\Q'\cup\bigcup_{\varphi\in\Phi^*}(\R')^\varphi$ in
$F/[R,F]$ generate the subgroup $(R\cap F')/[R,F]$. Similarly,
the their images in $F/[R\gamma_cF,F]$ generate the
subgroup $(R\cap F')[R\gamma_cF,F]/[R\gamma_cF,F]$. Since
$[R\gamma_cF,F]=[R,F]\gamma_{c+1}F$, we have that
\[
  (R\cap F')[R\gamma_cF,F]/[R\gamma_cF,F] 
  = (R\gamma_{c+1}F \cap F') /[R\gamma_cF,F].
\]
The latter subgroup is contained in $(R\gamma_cF\cap F')/[R\gamma_cF,F]$
which is isomorphic to the Schur multiplier $M(G/\gamma_cG)$.

As the group $G$ is invariantly $L$-presented, every $\varphi\in\Phi$
induces an endomorphism $\ti\varphi$ of $R\gamma_cF/[R\gamma_cF,F]$. This
yields, that the image of $M(G)$ in $M(G/\gamma_cG)$ has the form
\begin{equation} \label{eqn:ImgMG}
  \la q^{\bar p}, (r^{\bar\pi})\,^{\ti\varphi} \mid
  q\in\Q', r\in\R', \varphi\in\Phi^* \ra.
\end{equation}
This can be used to investigate the $\Phi^*$-module structure of $M(G)$
by considering the finitely generated Dwyer quotients $M_c(G)$. In our algorithm, we use
Hermite normal form computations in a spinning algorithm for computing a
finite generating set of the subgroup in~(\ref{eqn:ImgMG}).  We summarize
our algorithm as follows: Write $G=F/R$.
\begin{tabbing}
  {\scshape DwyerQuotient}( $G$, $c$ ) \\
  \qquad\= Compute an invariant finite $L$-presentation as in 
           Theorem~\ref{thm:ModLp}.\\[0.5ex]
  \>Compute a weighted nilpotent presentation for $G/\gamma_cG$ \\
  \>\qquad\=together with the natural homomorphism $F\to G/\gamma_cG$.\\[0.5ex]
  \>Compute a polycyclic presentation for the group
    $F/[R\gamma_cF,F]$\\
  \>\>together with the natural homomorphism $F\to F/[R\gamma_cF,F]$.\\[0.5ex]
  \>Translate each $\varphi\in\Phi$ to an endomorphism of the group
  $F/[R\gamma_cF,F]$\\
  \>\> and restrict this endomorphism to $(R\gamma_{c+1}F\cap F')/[R\gamma_cF,F]$.\\[0.5ex]
  \>Use the spinning algorithm to compute a finite generating set\\
  \>\> for the image $(R\gamma_{c+1}F\cap F')/[R\gamma_cF,F]$.
\end{tabbing}

\section{Applications}\label{sec:Apps}
The algorithm described in the first part is available in the
{\sf NQL}-package~\cite{NQL} of the computer algebra system {\sf
GAP}; see~\cite{GAP}. We parallelized the algorithm in~\cite{BEH08} to
enlarge the possible depths in the lower central series reached in this
section. We show the successful application of our algorithm to the
following invariantly finitely $L$-presented \emph{testbed groups}
studied in~\cite{Bar03} and~\cite{BEH08}:
\begin{itemize}\addtolength{\itemsep}{-0.75ex}
\item The Grigorchuk group $\Grig$, see~\cite{Gri80}, with its
      invariant finite $L$-pre\-sen\-ta\-tion from~\cite{Lys85}; see
      also~\cite{Gri99} and Example~\ref{ex:ModGri};
\item the twisted twin $\bar\Grig$ of the Grigorchuk group, see~\cite{BS09}, 
      with its invariant finite $L$-presentation from~\cite{BS09};
\item the Grigorchuk super-group $\ti\Grig$, see~\cite{BG02}, with
      its invariant finite $L$-pre\-sen\-ta\-tion from~\cite{Bar03};
\item the Basilica group $\Delta$, see~\cite{GZ02a,GZ02b}, with its
      invariant finite $L$-pre\-sen\-ta\-tion from~\cite{BV05}; and
\item the Brunner-Sidki-Vieira group $\BSV$, see~\cite{BSV99}, with
      its invariant finite $L$-pre\-sen\-ta\-tion from~\cite{Bar03}.
\end{itemize}
In Section~\ref{sec:FG}, we further applied our algorithm to
several generalized Fabry\-kowski-Gupta groups: an infinite family of
finitely $L$-presented groups $\FG_p$ introduced in~\cite{Gri00}. Invariant
finite $L$-pre\-sen\-ta\-tions for these groups were computed
in~\cite{BEH08}.\smallskip

\subsection{Aspects of the implementation of our algorithm in \GAP}
Table~\ref{tab:Perf} shows some performance data of the implementation of
our algorithm in the \NQL-package of the computer-algebra-system \GAP. All
timings displayed below have been obtained on an {\it Intel Pentium Core~2
Quad} with clock speed $2.83$~GHz using a single core. We applied our
algorithm with a time limit of two hours. Then the computations have
been stopped and the total time used to compute a weighted nilpotent
presentation for the quotient $G/\gamma_cG$ and the total time to compute
the Dwyer quotient $M_c(G)$ have been listed. Every application completed
within $1$~GB of memory.

\begin{table}[ht]
  \begin{center}
  \caption{Performance data of our implementation in \GAP}\medskip
  \label{tab:Perf}
  \begin{tabular}{cccc}
    \toprule
    & &\multicolumn{2}{c}{Time (h:min) for}\\
    \rb{$G$}&\rb{$c$}& $G/\gamma_{c+1}G$ & $M_{c+1}(G)$\\
    \midrule
    $\Grig$    & 90 & 1:47 & 0:07\\
    $\bar\Grig$& 54 & 1:44 & 0:09\\ 
    $\ti\Grig$ & 44 & 1:32 & 0:13\\
    $\Delta$   & 42 & 1:31 & 0:16\\
    $\BSV$     & 35 & 1:10 & 0:21\\
    $\FG_3$    & 75 & 1:46 & 0:04\\
    \bottomrule
  \end{tabular} \qquad
  \begin{tabular}{cccc}
    \toprule
    & &\multicolumn{2}{c}{Time (h:min) for}\\
    \rb{$G$}&\rb{$c$}& $G/\gamma_{c+1}G$ & $M_{c+1}(G)$\\
    \midrule
    $\FG_4$    & 71 & 1:50 & 0:07\\
    $\FG_5$    & 55 & 1:40 & 0:04\\
    $\FG_7$    & 46 & 1:40 & 0:03\\
    $\FG_8$    & 56 & 1:54 & 0:06\\
    $\FG_9$    & 61 & 1:44 & 0:06\\
    $\FG_{11}$ & 35 & 1:54 & 0:02\\
    \bottomrule
  \end{tabular}
\end{center}
\end{table}
We note that for the results shown in the remainder of this section we
used a parallel version of the algorithm for computing $G/\gamma_{c+1}G$.

\subsection{On the Dwyer quotients of the testbed-groups}
The Dwyer quotient $M_c(G)=M(G)/\ker\varphi_c$ is a finitely generated abelian
group and hence, it can be described by its abelian invariants or, if
the group is $p$-elementary abelian, by its $p$-rank. Here the list
$(c_1,\ldots,c_n)$ stands for the group $\Z_{c_1}\oplus\cdots\oplus
\Z_{c_n}$. For abbreviation, we will write $a^{[\ell]}$ if the term
$a$ occurs in $\ell$ consecutive places in a list. In the following 
we summarize our computational results for the testbed groups.\smallskip

The Grigorchuk group $\Grig$ was shown in~\cite{Gri80} to be an
explicit counter-example to the general Burnside problem: it is
a finitely generated infinite $2$-torsion group. Furthermore,
the Grigorchuk group is a first example of a group with an
intermediate word-growth. In~\cite{Lys85}, Lys\"enok determined a first
$L$-presentation for the group $\Grig$; see~Example~\ref{ex:ModGri}. Even
though it was already proposed in~\cite{Gri80} that the Grigorchuk
group $\Grig$ is not finitely presentable, a proof was not derived
until~\cite{Gri99} where Grigorchuk explicitly computed the Schur
multiplier of $\Grig$: it is infinitely generated $2$-elementary
abelian. We have computed the Dwyer quotients $M_c(\Grig)$, for $1\leq c\leq
301$. These quotients are $2$-elementary abelian with the following
$2$-ranks
\[
  1,2,3^{[3]},5^{[6]},7^{[12]}, 9^{[24]}, 11^{[48]}, 13^{[96]}, 15^{[110]}.
\]
This suggests the following conjecture.
\begin{conjecture}
  The Grigorchuk group $\Grig$ satisfies 
  \[
   M_c(\Grig) \cong \left\{\begin{array}{cl}
   \Z_2\textrm{ or }(\Z_2)^2,&\textrm{if }c=1\textrm{ or }c=2\textrm{, respectively}\\
   (\Z_2)^{2m+3},&\textrm{if }c\in\{3\cdot 2^m,\ldots,3\cdot 2^{m+1}-1\}
   \end{array}\right\},
  \]
  with $m\in\N_0$. 
\end{conjecture}
Further experiments suggest that the Schur multiplier of the Grigorchuk
group $\Grig$ has the $\{\sigma\}^*$-module structure, as given by
Lemma~\ref{lem:ModStruc}, of the form $\Z_2 \oplus (\Z_2[\sigma])^2$
where $\sigma$ fixes the first component.\medskip

The twisted twin $\bar\Grig$ of the Grigorchuk group was introduced
in~\cite{BS09}. It is invariantly finitely $L$-presented by
\[
  \la\{a,b,c,d\}\mid\{a^2, b^2, c^2, d^2\}\mid\{\bar\sigma\}
     \mid\{ [d^a,d], [d,c^ab], [d,(c^ab)^c], [d,(c^ab)^c], [c^ab,cb^a]\} \ra
\]
where $\ti\sigma$ is the free group endomorphism induced by the mapping
\[ 
  \ti\sigma\colon~\left\{\begin{array}{rcl}
  a &\mapsto& c^a \\
  b &\mapsto& d \\
  c &\mapsto& b^a \\
  d &\mapsto& c.
  \end{array}\right.
\]
We have computed the Dwyer quotients $M_c(\bar\Grig)$, for $1\leq c\leq 144$.
These quotients are $2$-elementary abelian with the following $2$-ranks
\[
  2,5,7,
  8^{[2]},   11^{[2]},
  12^{[4]},  15^{[4]},
  16^{[8]},  19^{[8]},
  20^{[16]}, 23^{[16]},
  24^{[32]}, 27^{[32]},
  28^{[17]}.
\]
This suggests the following conjecture.
\begin{conjecture}
  The twisted twin $\bar\Grig$ of the Grigorchuk group satisfies
  \[
   M_c(\bar\Grig) \cong \left\{\begin{array}{cl}
    (\Z_2)^{2},~(\Z_2)^{5},\textrm{ or }(\Z_2)^{7},&\textrm{if }c=1,~c=2,
    \textrm{ or }c=3,\textrm{ resp.}\\
    (\Z_2)^{4(m+1)+4},&\textrm{if }c\in\{2^{m+2},\ldots,2^{m+2}+2^{m+1}-1\}\\
    (\Z_2)^{4(m+1)+7},&\textrm{if }c\in\{2^{m+2}+2^{m+1},\ldots,2^{m+3}-1\}
   \end{array}\right\},
  \]
  with $m\in\N_0$. 
\end{conjecture}
Further experiments suggest that the Schur multiplier of
$\bar\Grig$ has the $\{\bar\sigma\}^*$-module structure, as given by
Lemma~\ref{lem:ModStruc}, of the form $(\Z_2[\bar\sigma])^4$; for a
proof see~\cite{BS09}.\medskip

The Grigorchuk super-group $\ti\Grig$ was introduced in~\cite{BG02}. It
contains the Grigorchuk group $\Grig$ as an infinite-index subgroup
and it is another example of a group with an intermediate word-growth.
In~\cite{Bar03}, it was shown that $\ti\Grig$ admits the invariant
finite $L$-presentation $\la\{\ti a,\ti b,\ti c,\ti d\}\mid
\emptyset\mid\{\ti\sigma\}\mid{\mc R}\ra$ where
\[
  {\mathcal R}=\big\{ a^2,[\tilde b,\tilde c],[\tilde c,\tilde c^a],
                      [\tilde c,\tilde d^a], [\tilde d,\tilde d^a], 
                      [\tilde c^{a\tilde b},(\tilde c^{a\tilde b})^a],
                      [\tilde c^{a\tilde b},(\tilde d^{a\tilde b})^a],
                      [\tilde d^{a\tilde b}, (\tilde d^{a\tilde b})^a]\big\}
\]
and $\ti\sigma$ is the free group endomorphism induced by the mapping
\[
  \tilde\sigma\colon\left\{
  \begin{array}{rcl}
           a&\mapsto&a\tilde b a\\
    \tilde b&\mapsto&\tilde d\\
    \tilde c&\mapsto&\tilde b\\
    \tilde d&\mapsto&\tilde c\,.
  \end{array}\right.
\]
The Schur multiplier of the group $\ti\Grig$ is still unknown. We have
computed the Dwyer quotients $M_c(\ti\Grig)$, for $1\leq c\leq 232$. These
quotients are $2$-elementary abelian with the following $2$-ranks
\[
  3,6,7,
   9^{[2]},  11^{[2]},
  13^{[4]},  15^{[4]},
  17^{[8]},  19^{[8]},
  21^{[16]}, 23^{[16]},
  25^{[32]}, 27^{[32]}, 
  29^{[64]}, 31^{[41]}.
\]
This suggests the following conjecture. 
\begin{conjecture}
  The Grigorchuk super-group $\ti\Grig$ satisfies
  \[
   M_c(\ti\Grig) \cong \left\{\begin{array}{cl}
   (\Z_2)^3,(\Z_2)^6,\textrm{ or }(\Z_2)^7,&\textrm{if }c=1,2,\textrm{ or }3\textrm{, respectively}\\
   (\Z_2)^{4m+5},&\textrm{if }c\in\{2\cdot 2^m,\ldots,3\cdot 2^{m}-1\}\\
   (\Z_2)^{4m+7},&\textrm{if }c\in\{3\cdot 2^m,\ldots,2\cdot 2^{m+1}-1\}
   \end{array}\right\},
  \]
  with $m\in\N$. 
\end{conjecture}
Further experiments suggest that the Schur multiplier of the
Grigorchuk super-group has the $\{\ti\sigma\}^*$-module structure,
as given by Lemma~\ref{lem:ModStruc}, of the form $(\Z_2)^3 \oplus
(\Z_2[\ti\sigma])^4$, where $\ti\sigma$ cyclically permutes the first
component.\medskip

The Basilica group $\Delta$ was introduced in~\cite{GZ02b,GZ02a} as
a torsion-free group defined by a three-state automaton.  Bartholdi
and Vir\'ag~\cite{BV05} computed the following invariant finite
$L$-pre\-sen\-ta\-tion:
\[
  \Delta \cong \la\{a,b\}\mid \emptyset \mid\{\sigma\}\mid\{[a,a^b]\}\ra
\]
where $\sigma$ is the free group endomorphism induced by the mapping
\[
  \sigma\colon\left\{
  \begin{array}{rcl}
   a &\mapsto& b^2\\
   b &\mapsto& a.
  \end{array}\right.
\]
We have computed the Dwyer quotients $M_c(\Delta)$, for $1\leq c\leq 103$. These
quotients satisfy the following conjecture.
\begin{conjecture}
  The Basilica group $\Delta$ satisfies
  \[
    M_c(\Delta) \cong \Z^2\oplus\bigoplus_{\ell\in\N}{\mc A}_\ell(c),
    \qquad\textrm{for each }c\geq 2,
  \]
  where the groups ${\mc A}_\ell(c)$ are given as follows: 
  \[
    {\mc A}_1(c) = \left\{\begin{array}{cl}
     0, & \textrm{if }c\in\{1,\ldots,5\}\\
     \Z_{2^{2(m+1)}},& \textrm{if }c\in\{2m+6,2m+7\}
    \end{array}\right\}
  \]
  and 
  \[
    {\mc A}_\ell(c) = \left\{\begin{array}{cl}
         0,&\textrm{if }c\in\{1,\ldots,3\cdot 2^{\ell+1}-1\}\\
    \Z_{2^{2m+1}},&\textrm{if }
    c\in\{(3+m)\cdot 2^{\ell+1},\ldots,(3+m)\cdot 2^{\ell+1}+2^{\ell-1}-1\}\\
    \Z_{2^{2m+2}},&\textrm{if }
    c\in\{(3+m)\cdot 2^{\ell+1}+2^{\ell-1},\ldots,(4+m)\cdot 2^{\ell+1}-1\}
  \end{array}\right\}
  \]
  with $m\in\N_0$.
  Hence, the Basilica group $\Delta$ is not finitely
  presentable.
\end{conjecture}\medskip

The Brunner-Sidki-Vieira group $\BSV$ was introduced in~\cite{BSV99}
as a just-non-solvable, torsion-free group acting on the binary tree. The
authors also gave the following invariant finite $L$-presentation:
\[
 \BSV \cong \la\{a,b\}\mid\emptyset\mid\{\varepsilon\}
                      \mid\{[b,b^a],[b,b^{a^3} ]\}\ra
\]
where $\varepsilon$ is the free group endomorphism induced by the mapping
\[
  \varepsilon\colon\left\{
  \begin{array}{rcl}
    a & \mapsto & a^2\\
    b & \mapsto & a^2\,b^{-1}\,a^2.
  \end{array}\right.
\]
We have computed the Dwyer quotients $M_c(\BSV)$, for $1\leq c\leq 53$.
These quotients satisfy the following conjecture.
\begin{conjecture}
  The Brunner-Sidki-Vieira group $\BSV$ satisfies 
  \[
    M_c(\BSV) \cong \Z^2 \oplus {\mc A}(c)
    \oplus \bigoplus_{\ell\in\N} {\mc B}_\ell(c)
    \oplus \bigoplus_{\ell\in\N} {\mc C}_\ell(c),
    \qquad\textrm{for each }  c\geq 2,
  \] 
  where the groups ${\mc A}(c)$, ${\mc B}_\ell(c)$, and ${\mc C}_\ell(c)$
  are given as follows: 
  \begin{eqnarray*}
    {\mc A}(c)&=&\left\{ \begin{array}{cl}
      0,&\textrm{if }c\in\{1,\ldots,3\} \\
    \Z_{2^{2m+1}},&\textrm{if }c\in\{2m+4, 2m+5 \}
    \end{array}\right\}
  \end{eqnarray*}
  with $m\in\N_0$. Additionally, for each $\ell\in\N$, we have
  \begin{eqnarray*}
    {\mc B}_\ell(c)&=&\left\{ \begin{array}{cl}
      0,&\textrm{if }c\in\{1,\ldots,5\cdot 2^{\ell-1}-1\} \\
    \Z_{2^{4m+1}},&\textrm{if }c\in\{2^{\ell+2}m+5\cdot 2^{\ell-1},\ldots,
                                    2^{\ell+2}m+6\cdot 2^{\ell-1}-1\}\\
    \Z_{2^{4m+2}},&\textrm{if }c\in\{2^{\ell+2}m+6\cdot 2^{\ell-1},\ldots,
                                    2^{\ell+2}m+10\cdot 2^{\ell-1}-1\}\\
    \Z_{2^{4m+4}},&\textrm{if }c\in\{2^{\ell+2}m+10\cdot 2^{\ell-1},\ldots,
                                    2^{\ell+2}m+13\cdot 2^{\ell-1}-1\}\\
    \end{array}\right\}
  \end{eqnarray*}
  and
  \begin{eqnarray*}
    {\mc C}_\ell(c)&=&\left\{ \begin{array}{cl}
      0,&\textrm{if }c\in\{1,\ldots,9\cdot 2^{\ell-1}-1\} \\
    \Z_{2^{4m+1}},&\textrm{if }c\in\{2^{\ell+2}m+9\cdot 2^{\ell-1},\ldots,
                                    2^{\ell+2}m+12\cdot 2^{\ell-1}-1\}\\
    \Z_{2^{4m+2}},&\textrm{if }c\in\{2^{\ell+2}m+12\cdot 2^{\ell-1},\ldots,
                                    2^{\ell+2}m+14\cdot 2^{\ell-1}-1\}\\
    \Z_{2^{4m+3}},&\textrm{if }c\in\{2^{\ell+2}m+14\cdot 2^{\ell-1},\ldots,
                                    2^{\ell+2}m+16\cdot 2^{\ell-1}-1\}\\
    \Z_{2^{4m+4}},&\textrm{if }c\in\{2^{\ell+2}m+16\cdot 2^{\ell-1},\ldots,
                                    2^{\ell+2}m+17\cdot 2^{\ell-1}-1\}\\
    \end{array}\right\}
  \end{eqnarray*}
  with $m\in\N_0$.
  Hence, the Brunner-Sidki-Vieira group $\BSV$ is not finitely
  presentable.
\end{conjecture}

\subsection{On the Dwyer quotients of some Fabrykowski-Gupta groups}
\label{sec:FG}
The Fabrykowski-Gupta group $\FG_3$ was introduced in~\cite{FG85} as an
example of a group with an intermediate word-growth. For every positive
integer $d$, Grigorchuk~\cite{Gri00} described a generalization $\FG_d$
of the Fabrykowski-Gupta group $\FG_3$. A rather longish invariant finite
$L$-presentation was computed in~\cite{BEH08}. Further, it was shown
that, if $d$ is not a prime-power, the group $\FG_d$ has a maximal
nilpotent quotient. This latter quotient is isomorphic to the maximal
nilpotent quotient of the wreath product $\Z_d\wr\Z_d$. We therefore
consider only those groups $\FG_d$ which admit a `rich' lower central
series; that is, the index $d$ is a prime-power.\medskip

Let $d\in\{3,5,7,11\}$ be a prime. Then the Dwyer quotients $M_c(\FG_d)$ are
$d$-elementary abelian with the following $d$-ranks.
\begin{center}
  \begin{tabular}{cl@{\,}l@{\,}l@{\,}l@{\,}l@{\,}l@{\,}l@{\,}l@{\,}l@{\,}l@{\,}}
    \toprule
    $d$ &\multicolumn{10}{c}{$\rk(M_c(\FG_d))$} \\
    \midrule
    $3$&$0^{[2]}$,&$1^{[3]}$,& $2^{[0]}$,& $3^{[9]}$,& $4^{[1]}$, &$5^{[26]}$,& $6^{[4]}$, &$7^{[77]}$,& $8^{[13]}$, &$9^{[12]}$ \\ 
    $5$&$0^{[1]}$, &$1^{[4]}$, &$2^{[2]}$, &$3^{[20]}$, &$4^{[10]}$, &$5^{[100]}$&$6^{[2]}$&&& \\ 
    $7$&$0^{[1]}$,& $1^{[2]}$, &$2^{[6]}$,&$ 3^{[2]}$,&$ 4^{[14]}$, &$5^{[42]}$, &$6^{[14]}$,&$7^{[34]}$ &&\\ 
    $11$&$0^{[1]}$, &$1^{[2]}$,&$2^{[2]}$,& $3^{[2]}$,& $4^{[10]}$, &$5^{[2]}$, &$6^{[22]}$, &$7^{[22]}$, &$8^{[22]}$, &$9^{[27]}$ \\ 
    \bottomrule
  \end{tabular}
\end{center}
As noted by Laurent Bartholdi and Olivier Siegenthaler, there is a
pattern in the ranks of the Dwyer quotients $M_c(\FG_d)$. For example,
it may holds that
\[
  M_c(\FG_5) \cong 
  \left\{\begin{array}{cl} 
   0,&\textrm{ if }c=0\\
   \Z_5^{2m+1},&\textrm{ if }c\in\{2+\frac{3}{2}(5^m-1),\ldots,
                                   1+\frac{3}{2}(5^m-1)+4\cdot 5^m\}\\[0.75ex]
   \Z_5^{2m+2},&\textrm{ if }c\in\{2+\frac{3}{2}(5^m-1)+4\cdot 5^m,
                            \ldots,1+\frac{3}{2}(5^{m+1}-1)\}
  \end{array}\right\}
\]
for $m\in\N_0$. This suggests the following conjecture.
\begin{conjecture}
  Let $d$ be a prime. Then the Schur multiplier of $\FG_d$, modulo the
  Dwyer-kernel, is infinitely generated $d$-elementary abelian.
\end{conjecture}
Finally, we summarize our results for $M_c(\FG_d)$ for
$d\in\{4,8,9\}$. The abelian invariants of the Dwyer quotients $M_c(\FG_d)$
are as follows.
\begin{center}
  \begin{tabular}{lr}
    \hline
    \multicolumn{1}{c}{$d$} & \multicolumn{1}{c}{$M_c(\FG_d)$} \\
    \hline
    & \raisebox{0ex}[2.5ex]{}
    $(1)^{[1]}$ 
    $(2)^{[1]}$ 
    $(2,2)^{[1]}$ 
    $(2,4)^{[4]}$ 
    $(2,2,2,4)^{[1]}$ \\ 4 &
    $(2,2,2,2,4)^{[4]}$  
    $(2,2,2,4,4)^{[16]}$ 
    $(2,2,2,2,4,4)^{[1]}$ 
    $(2,2,2,2,2,4,4)^{[3]}$  \\ &
    $(2,2,2,2,2,2,4,4)^{[16]}$ 
    $(2,2,2,2,2,4,4,4)^{[64]}$ 
    $(2,2,2,2,2,2,4,4,4)^{[5]}$ \\ &
    $(2,2,2,2,2,2,2,4,4,4)^{[11]}$ 
    $(2,2,2,2,2,2,2,2,4,4,4)^{[26]}$ \\
    \hline 
    & \raisebox{0ex}[2.5ex]{}
    $(1) ^ {[1]}$
    $(8) ^ {[2]}$
    $(4,8) ^{[3]}$
    $(2,4,8) ^{[4]}$
    $(2,8,8) ^{[1]}$
    $(2,2,8,8) ^{[2]}$\\ &
    $(2,2,2,8,8) ^{[2]}$
    $(2,2,4,8,8) ^{[2]}$
    $(2,4,4,8,8) ^{[2]}$
    $(2,4,8,8,8) ^{[2]}$\\ \raisebox{1.5ex}[-1.5ex]{8} &
    $(2,8,8,8,8) ^{[8]}$
    $(2,2,8,8,8,8) ^{[4]}$
    $(2,4,8,8,8,8) ^{[20]}$
    $(2,2,4,8,8,8,8) ^{[32]}$\\ &
    $(2,2,8,8,8,8,8) ^{[7]}$
    $(2,2,2,8,8,8,8,8) ^{[16]}$
    $(2,2,2,2,8,8,8,8,8) ^{[16]}$\\ & 
    $(2,2,2,4,8,8,8,8,8) ^{[16]}$
    $(2,2,4,4,8,8,8,8,8) ^{[3]}$\\
    \hline
    & \raisebox{0ex}[2.5ex]{}
    $(1) ^ {[1]}$
    $(9) ^ {[2]}$
    $(3,9) ^ {[2]}$
    $(3,3,9) ^ {[4]}$
    $(3,9,9) ^ {[2]}$ \\ &
    $(9,9,9) ^ {[2]}$
    $(3,9,9,9) ^ {[2]}$
    $(3,3,9,9,9) ^ {[4]}$
    $(3,9,9,9,9) ^ {[2]}$ \\ \raisebox{1.5ex}[-1.5ex]{9} &
    $(9,9,9,9,9) ^ {[12]}$
    $(3,9,9,9,9,9) ^ {[18]}$
    $(3,3,9,9,9,9,9) ^ {[36]}$ \\ &
    $(3,9,9,9,9,9,9) ^ {[18]}$
    $(9,9,9,9,9,9,9) ^ {[17]}$
    $(3,9,9,9,9,9,9,9) ^ {[12]}$\\
    \hline
  \end{tabular}
\end{center}
Again, these computational results suggest that the groups $\FG_d$
are not finitely presentable. Further, the exponent of $M_c(\FG_d)$
is most likely the index $d$ itself.

\def\cprime{$'$}

\bigskip

\noindent Ren\'e Hartung,
{\scshape Mathematisches Institut},
{\scshape Georg-August Universit\"at zu G\"ottingen},
{\scshape Bunsenstra\ss e 3--5},
{\scshape 37073 G\"ottingen}
{\scshape Germany}\\[1ex]
{\it Email:} \qquad \verb|rhartung@uni-math.gwdg.de|\\[2.ex]
June 2009 (revised March 2010, accepted April 2010).

\end{document}